\documentclass[12pt,a4paper]{amsart}
\usepackage{mathrsfs}
\usepackage{amssymb,amsmath,amsthm,color}
\usepackage{graphicx,mcite}
\usepackage{hyperref}
\usepackage{url}
\usepackage{setspace}
\usepackage{enumerate}
\newtheorem{theorem}{Theorem}[section]
\newtheorem{lemma}[theorem]{Lemma}
\newtheorem{proof of lemma}[theorem]{Proof of Lemma}
\newtheorem{proposition}[theorem]{Proposition}

\theoremstyle{definition}
\newtheorem{definition}[theorem]{Definition}

\newtheorem{remark}[theorem]{Remark}

\numberwithin{equation}{section}

\begin{document}

\title[Heisenberg uniqueness pair]
{Heisenberg uniqueness pairs for the Fourier transform on the Heisenberg group}
%\runningtitle{Sets of Injectivity for Weighted Twisted Spherical Means}

\author{Somnath Ghosh and R.K. Srivastava}
 %Address of record for the research reported here

\address{Department of Mathematics, Indian Institute of Technology, Guwahati, India 781039.}
%\curraddr{Department of Mathematics and Statistics,
%Case Western Reserve University, Cleveland, Ohio 43403}
\email{gsomnath@iitg.ac.in, rksri@iitg.ac.in}
%\thanks will become a 1st page footnote.

%    General info
\subjclass[2000]{Primary 42A38; Secondary 44A35}

\date{\today}

\keywords{Convolution, Fourier transform, Heisenberg group, Laguerre polynomial.}

\begin{abstract}
In this article, we prove that (unit sphere, non-harmonic cone) is a Heisenberg
uniqueness pair for the symplectic Fourier transform on $\mathbb C^n.$ We
derive that spheres as well as non-harmonic cones are determining sets for the
spectral projections of the finite measure supported on the unit sphere.
Further, we prove that if the Fourier transform of a finitely supported
function on step two nilpotent Lie group is of arbitrary finite rank,
then the function must be zero. The latter result correlates to the
annihilating pair for the Weyl transform.
\end{abstract}

\maketitle

\section{Introduction}\label{section1}

Let $\Gamma$ be a finite disjoint union of smooth curves in the plane $\mathbb R^2.$
Let $X(\Gamma)$ be the space of all finite complex-valued Borel measures
$\mu$ in $\mathbb R^2,$ which are supported on $\Gamma$ and absolutely
continuous with respect to the arc length measure of $\Gamma$. For
$(\xi,\eta)\in\mathbb R^2,$ the Fourier transform of $\mu$ can be defined by
\[\hat\mu{( \xi,\eta)}=\int_\Gamma e^{-i(x\cdot\xi+ y\cdot\eta)}d\mu(x,y).\]
In the above context, $\hat\mu$ becomes uniformly continuous bounded function.
Thus, we can analyze the pointwise vanishing property of $\hat\mu.$
\begin{definition}
Let $\Lambda$ be a set in $\mathbb R^2.$ The pair $\left(\Gamma, \Lambda\right)$
is called a Heisenberg uniqueness pair for $X(\Gamma)$ if the only $\mu\in X(\Gamma)$
that satisfies $\hat\mu\vert_\Lambda=0$ is $\mu=0.$
\end{definition}

\smallskip

In general, the problem of Heisenberg uniqueness pair (HUP) is a question
about the determining property of the finite Borel measures, which are supported
on some given lower dimensional entities whose Fourier transform also vanishes on lower
dimensional entities. In particular, if $\Gamma$ is compact, then by Paley Wiener
theorem $\hat\mu$ is real analytic having exponential growth, and hence $\hat\mu$
can vanish only on a set of measure zero. Thus, the HUP problem becomes a little
easier in this case. However, this problem becomes immensely difficult while the
measure is supported on a non-compact entity. Further, it appears that the problem
of HUP is a natural variant of the uncertainty principle for the Fourier
transform.

\smallskip

In addition, the concept of determining the Heisenberg uniqueness pair has a
correlation with the concept of annihilating pair. Let $A\subseteq \mathbb{R}$ and
$\Sigma\subseteq \hat{\mathbb{R}}$ be measurable subsets. Then the pair $(A,\Sigma)$
is called a {\em weak annihilating pair} (WAP) if $\text{supp}f\subseteq A$ and
$\text{supp}\hat{f}\subseteq \Sigma,$ implies $f=0.$ The pair $(A,\Sigma)$
is called a {\em strong annihilating pair} (SAP) if there exists a positive number
$C=C(A,\Sigma)$ such that
\begin{align}\label{exp29}
\|f\|_2^2\leq C\left(\|f\|_{L^2{(A^c)}}+\|\hat{f}\|_{L^2{(\Sigma^c)}}\right)
\end{align}
for all $f\in L^2(\mathbb{R}).$ It is obvious from (\ref{exp29}) that every SAP is
a WAP. In \cite{B}, Benedicks had, eventually, proved that $(A,\Sigma)$ is a WAP
if $A$ and $\Sigma$ both have finite measure, whereas it follows from Amrein-Berthier
\cite{AB} that $(A,\Sigma)$ is a SAP under the identical assumption as in \cite{B}.
For the latter conclusion we refer to \cite{HJ}. Later, various analogues of
Benedicks-Amrein-Berthier theorem and qualitative uncertainty principle,
have been investigated in different setups (see \cite{BD, FS, HJ, H, NR, PS, SST}).

\smallskip

In \cite{NR}, Narayanan and Ratnakumar proved that if $f\in L^1(\mathbb{H}^n)$ is
supported on $B\times \mathbb{R}$, where $B$ is a compact subset of $\mathbb{C}^n$,
and if $\hat{f}(\lambda)$ is of finite rank for each $\lambda,$ then $f=0.$ Thereafter,
Vemuri \cite{V} replaced the compactness condition on $B$ by finite measure.
In \cite{CGS}, authors considered $B$ as a rectangle in $\mathbb R^{2n}$ to prove
analogous result in step two nilpotent Lie group.

\smallskip

In this article, we prove the result on the general step two nilpotent Lie group
when $B$ is an arbitrary set of finite measure, using the Hilbert space theory.
Though specifying the appropriate set of projections in the general setups was
a major bottleneck, and we sorted out the same. For the sake of simplicity, we prefer
to prove the result for the case of the Heisenberg group as the same technique can be
extended for the general step two nilpotent Lie group. Finally, we consider
the case of SAP for the Weyl transform and draw some conclusions therein.

\smallskip

We discuss the concept of HUP, which was introduced by Hedenmalm
and Montes-Rodr\'iguez in \cite{HR}, and then inn \cite{HR}, they have shown that the
pair (hyperbola, some discrete set) is a Heisenberg uniqueness pair. As a dual
problem, a weak$^\ast$ dense subspace of $L^{\infty}(\mathbb R)$  has been
constructed to solve the long standing  Klein-Gordon equation. Further,
Hedenmalm and Montes-Rodr\'iguez \cite{HR} have given a complete characterization
of the Heisenberg uniqueness pairs corresponding to any two parallel lines.

\smallskip

Lev \cite{L} and Sj\"{o}lin \cite{S1} have independently shown that
circle and certain system of lines form HUP corresponding to the unit circle.
Further, Vieli \cite{V1} has generalized the case of circle in the higher
dimension and shows that a sphere whose radius does not lie in the zero sets of the
Bessel functions $J_{(n+2k-2)/2};~k\in\mathbb Z_+,$ the set of non-negative integers,
is a HUP corresponding to the unit sphere $S^{n-1}.$ In \cite{Sri5}, the  author
has shown that a cone is a HUP corresponding to the sphere as long
as the cone does not completely recline on the level surface of any homogeneous harmonic
polynomial on $\mathbb R^n.$

\smallskip

Further, Sj\"{o}lin \cite{S2} derived some Heisenberg uniqueness pairs corresponding to
the parabola. It has been extended to the case of paraboloid by Vieli \cite{V2}. Subsequently,
Babot \cite{Ba} has given a characterization of the Heisenberg uniqueness pairs corresponding
to a certain system of three parallel lines. Thereafter, the authors in \cite{GR} have given
some necessary and sufficient conditions for the Heisenberg uniqueness pairs corresponding to
a system of four parallel lines. However, the question of the unique necessary and sufficient
condition for the finitely many parallel lines as compared to three lines result \cite{Ba} is
still unsolved. For more examples of HUP corresponding to spiral, exponential
curves, sphere and paraboloid, see \cite{CGGS, GR}.

\smallskip

Jaming and Kellay \cite{JK} have given a unifying proof
for some of the Heisenberg uniqueness pairs corresponding to the hyperbola, polygon,
ellipse and graph of the functions $\varphi(t)=|t|^\alpha,$ whenever $\alpha>0.$
Thereafter, Gr\"{o}chenig and Jaming \cite{GJ} have worked out some of the Heisenberg
uniqueness pairs corresponding to quadratic surfaces. For more results on HUP via
ergodic theory dynamical systems, see \cite{CHM, HR2, HR3, GRR}.

\smallskip

In this article, we work out a few analogous results on the Heisenberg group
in various aspects. Firstly, we consider the symplectic Fourier transform (SFT)
on $\mathbb C^n.$ We prove that $S^{2n-1}$ forms a HUP with a non-harmonic
complex cone for SFT. The above result has a sharp contrast with analogous
result for the Euclidean Fourier transform (EFT) on $\mathbb R^{2n}.$
Since a non-trivial complex cone in $\mathbb C^n~(n\geq2)$ can have topological
dimension at most $2n-2,$ it follows that a $(2n-2)$ - dimensional entity forms
a HUP with $S^{2n-1}$ for SFT. Although for EFT on $\mathbb R^{2n},$ the least
topological dimension required (in general) for a set to be a HUP with unit sphere
$S^{2n-1}$ is $2n-1.$ We also observe that the conclusion of the above result
for the SFT holds true for a real non-harmonic cone in $\mathbb R^{2n}.$

\smallskip

Thereafter, we consider the case of modified Fourier transform on the Heisenberg
group. We prove that a finite measure supported on the cylinder $S^{2n-1}\times\mathbb R$
can be determined by a non-harmonic cone as well as the boundary of a bounded domain in
$\mathbb C^n.$

\smallskip

Further, we consider reasonably more interesting case of determining a finite
measure $\mu$ which is supported on $S^{2n-1},$ in terms of its spectral projections.
We prove that if the spectral projections $\varphi_k^{n-1}\times\mu$ vanish on
a sphere of arbitrary radius, then $\mu$ is trivial. We observed that the above
measures can also be determined by a non-harmonic complex cone. Though, the case
of the real non-harmonic cone is yet to be settled, we leave it open for the time
being.

\section{Preliminaries}\label{section2}
In this section, we describe some basic facts about Fourier transform on
the Heisenberg group, Weyl transform, special Hermite functions and expansion
of functions on $\mathbb C^n$ accordingly. Subsequently, we mention some auxiliary
results for the bigraded spherical harmonics and non-harmonic real as well as
complex cone.

The Heisenberg group $\mathbb H^n=\mathbb C^n\times\mathbb R$ is a step two nilpotent
Lie group having center $\mathbb R$ that equipped with the group law
\[(z,t)\cdot(w,s)=\left(z+w,t+s+\frac{1}{2}\text{Im}(z\cdot\bar w)\right).\]

\smallskip

By Stone-von Neumann theorem, the infinite dimensional irreducible unitary
representations of $\mathbb H^n$ can be parameterized by $\mathbb R^\ast=\mathbb R\smallsetminus\{0\}.$
That is, each of $\lambda\in\mathbb R^\ast$ defines a Schr\"{o}dinger representation $\pi_\lambda$ of
$\mathbb H^n$ by
\[\pi_\lambda(z,t)\varphi(\xi)=e^{i\lambda t}e^{i\lambda(x\cdot\xi+\frac{1}{2}x\cdot y)}\varphi(\xi+y),\]
where $z=x+iy$ and $\varphi\in L^2(\mathbb{R}^n).$ Hence, the group Fourier transform of
$f\in L^1(\mathbb H^n)$ can be defined  by
\[\hat f(\lambda)=\int_{\mathbb H^n}f(z,t)\pi_\lambda(z,t)dzdt\]
\smallskip
is a bounded operator. When $f\in L^2(\mathbb{H}^n), $ $\hat f(\lambda)$ is a Hilbert-Schmidt operator.
An important technique in
many problems on $\mathbb H^n$ is to take partial Fourier transform in the $t$-variable
to reduce the matter to $\mathbb C^n$. Let
\[f^\lambda(z)=\int_\mathbb R f(z,t)e^{i \lambda t} dt\]
be the inverse Fourier transform of $f$ in the $t$-variable. The group
convolution of the functions $f, g\in L^1(\mathbb H^n),$ defined by
\begin{equation} \label{exp22}
f\ast g(z, t)=\int_{H^n}~f((z,t)(-w,-s))g(w,s)~dwds,
\end{equation}
by taking the inverse Fourier transform in the $t$-variable, takes the form
\begin{eqnarray*}
(f \ast g)^\lambda(z)&=&\int_{-\infty}^{~\infty}~f \ast g(z,t)e^{i\lambda t} dt\\
&=&\int_{\mathbb C^n}~f^\lambda (z-w)g^\lambda(w)e^{\frac{i\lambda}{2}\text{Im}(z.\bar{w})}~dw\\
&=&f^\lambda\times_\lambda g^\lambda(z),
\end{eqnarray*}
where $f^\lambda\times_\lambda g^\lambda$ is known as $\lambda$-twisted convolution.
Thus, the group convolution  $f\ast g$ on the Heisenberg group can be studied by
the $\lambda$-twisted convolution $f^\lambda\times_\lambda g^\lambda$ on $\mathbb C^n.$
When $\lambda \neq 0,$ by a scaling argument, it is enough to study the twisted convolution
for the case $\lambda=1.$

\smallskip

Now, we recall the Weyl transform, which is the most non-commutative
constituent of the group Fourier transform on the Heisenberg group.
Denote by $\pi_\lambda(z)=\pi_\lambda(z,0).$
Then $\pi_\lambda(z,t)=e^{i\lambda t}\pi_\lambda(z).$ For a suitable function $g$ on
$\mathbb{C}^n$, the Weyl transform of $g$ can be expressed by
\[W_\lambda(g)=\int_{\mathbb C^n}g(w)\pi_\lambda(w)dw.\]
This, in turn, implies $\hat f(\lambda)=W_\lambda(f^\lambda).$  It is easy
followed that $W_\lambda(g)$ is a bounded operator, whenever $g\in L^1(\mathbb{C}^n).$
On the other hand, if $g\in L^2(\mathbb{C}^n),$
$W_\lambda(g)$ is a Hilbert-Schmidt opertor and satisfies the Plancherel formula
\[|\lambda|^{\frac{n}{2}}\|W_\lambda(g)\|_{HS}=(2\pi)^{\frac{n}{2}}\|g\|_2.\]
The Fourier-Winger transform of $\varphi,\psi \in L^2(\mathbb{R}^n)$ is defined by
the formula $V_{\varphi}^{\psi}(z)=(2\pi)^{-n/2}\langle \pi(z)\varphi,\psi \rangle.$
It is known that
$V_{\varphi}^{\psi}\in L^2(\mathbb{C}^n)$ and satisfies the identity
\begin{align}\label{exp28}
\int_{\mathbb{C}^n}V_{\varphi_1}^{\psi_1}(z)\overline{V_{\varphi_2}^{\psi_2}(z)}dz
=\langle \varphi_1,\varphi_2\rangle \overline{\langle \psi_1,\psi_2\rangle},
\end{align}
whenever $\varphi_l,\psi_l \in L^2(\mathbb{R}^n); \,l=1,2.$ See \cite{T1}.

\smallskip

Next, we describe the special Hermite expansion for function on $\mathbb C^n.$
Let \[T=\frac{\partial}{\partial t},~ X_j=\frac{\partial}{\partial
x_j}+\frac{1}{2}y_j\frac{\partial}{\partial t}~\text{and}~
Y_j=\frac{\partial}{\partial y_j}-\frac{1}{2}x_j\frac{\partial}{\partial t}\]
be the left-invariant vector fields on $\mathbb H^n.$ Then $\{T,X_j,Y_j: j=1,\ldots,n\}$
forms a basis for the Lie algebra $\mathfrak h^n$  of $\mathbb H^n,$ and the representation
$\pi_\lambda$ induces a representation $\pi_\lambda^*$ of $\mathfrak h^n$ on the space
of $C^\infty$ vectors in $L^2(\mathbb R^n)$ via
\[\pi_\lambda^*(X)f=\left.\frac{d}{dt}\right\vert_{t=0}\pi_\lambda(\exp tX)f.\]

\smallskip

It is easy to see that $\pi_\lambda^*(X_j)=i\lambda x_j$ and $\pi_\lambda^*(Y_j)=\frac{\partial}{\partial x_j}.$
Hence for the sub-Laplacian $\mathcal L=-\sum_{j=1}^n(X_j^2+Y_j^2),$
it follows that $\pi_\lambda^*(\mathcal L)=-\Delta_x+\lambda^2|x|^2=:H_\lambda,$
the scaled Hermite operators. Let
$\phi_\alpha^\lambda(x)=|\lambda|^{\frac{n}{4}}\phi_\alpha(\sqrt{|\lambda|}x);~\alpha\in\mathbb Z_+^n,$
where $\phi_\alpha$'s are the Hermite functions on $\mathbb R^n.$
Then each $ \phi_\alpha^\lambda$ is an eigenfunction of $H_\lambda$ with eigenvalue
$(2|\alpha|+n)|\lambda|.$ Hence the entry functions $E_{\alpha\beta}^\lambda$'s
of the representation $\pi_\lambda$ are eigenfunctions of the sub-Laplacian $\mathcal L$
satisfying \[\mathcal L E_{\alpha\beta}^\lambda=(2|\alpha|+n)|\lambda|E_{\alpha\beta}^\lambda,\]
where $E_{\alpha\beta}^\lambda(z,t)=\left\langle\pi_\lambda(z,t)\phi_\alpha^\lambda,\phi_\beta^\lambda\right\rangle.$
Since $E_{\alpha\beta}^\lambda(z,t)=e^{i\lambda t}\left\langle\pi_\lambda(z)\phi_\alpha^\lambda,\phi_\beta^\lambda\right\rangle,$
the eigenfunctions $E_{\alpha\beta}^\lambda$'s are not in $L^2(\mathbb H^n).$
However, for each fixed $t,$ they are in $L^2(\mathbb C^n).$ Now, define an operator
$L_\lambda$ by $\mathcal L\left(e^{i\lambda t}f(z)\right)=e^{i\lambda t}L_\lambda f(z).$
Then the special Hermite function
\[\phi_{\alpha\beta}^\lambda(z)=(2\pi)^{-\frac{n}{2}}\left\langle\pi_\lambda(z)\phi_\alpha^\lambda,\phi_\beta^\lambda\right\rangle\]
is an eigenfunction of $L_\lambda$ with eigenvalue $2|\alpha|+n.$ Now,
we can summarize that the special Hermite functions $\phi^\lambda_{\alpha\beta}$'s
form an orthonormal basis for $L^2(\mathbb C^n)$ \cite{T2}. Hence
$g\in L^2(\mathbb C^n)$ can be expressed as
\begin{equation}\label{exp31}
g=\sum_{\alpha,\beta}\left\langle g,\phi^\lambda_{\alpha\beta}\right\rangle\phi^\lambda_{\alpha\beta}.
\end{equation}
By employing a correlation of the special Hermite functions with the Laguerre function,
expression (\ref{exp31}) can be further simplified, for which we need to recall
the definition of Laguerre function. Given $\upsilon\in\mathbb{C},$ the Laguerre polynomial of
degree $k\in\mathbb{Z}_+$ is defined by
\begin{align*}
L_k^\upsilon(x)=\sum_{j=0}^k \binom{\upsilon+k}{k-j}\frac{(-x)^j}{j!}.
\end{align*}
Now, Laguerre function on $\mathbb{C}^n$ of order $n-1$ and degree $k$
can be defined by $\varphi_k^{n-1}(z)=L_k^{n-1}(\frac{|z|^2}{2}) e^{-\frac{|z|^2}{4}}.$ Denote
$\varphi^{n-1}_{k,\lambda}(z)=\varphi^{n-1}_k(\sqrt{|{\lambda|}}z),$ where $\lambda\in \mathbb{R}^*.$
Then the special Hermite functions $\phi^\lambda_{\alpha\alpha}$ will satisfy the relation
\begin{equation}\label{ACexp4}
\sum_{|\alpha|=k}\phi^\lambda_{\alpha,\alpha}(z)=(2\pi)^{-\frac{n}{2}}
|\lambda|^{\frac{n}{2}}\varphi^{n-1}_{k,\lambda}(z).
\end{equation}
Thus, $g\in L^2(\mathbb C^n)$ can be expressed by
\[g(z)=(2\pi)^{-n}|\lambda|^n\sum_{k=0}^\infty g\times_\lambda\varphi_{k,\lambda}^{n-1}(z).\]
Please refer to \cite{T2}. In particular, for $\lambda=1$, we have
\begin{equation}\label{ACexp16}
g(z)=(2\pi)^{-n}\sum_{k=0}^\infty g\times\varphi_k^{n-1}(z),
\end{equation}
which is the special Hermite expansion of $g$. Hence $g$ can be completely
determined by its spectral projections $g\times\varphi_k^{n-1}$. Thus, it is
an interesting question to determine finite measures $\mu$ those are
supported on a thin set in $\mathbb C^n$ via spectral projections
$\varphi_k^{n-1}\times\mu.$ We discuss this assertion in Section \ref{section5}.

\smallskip

Let $P_{p,q}$ denote the space of all bi-graded homogeneous polynomials on
$\mathbb C^n$ of the form
\begin{equation}\label{exp01}
P(z)=\sum_{|\alpha|=p}\sum_{|\beta|=q}c_{\alpha\beta}z^\alpha\bar{z}^\beta,
\end{equation}
where $p,q\in\mathbb Z_+.$ Denote $H_{p,q}=\{P\in P_{p,q}:\Delta P=0\}.$
The restriction of bi-graded homogeneous harmonic polynomial to the unit
sphere $S^{2n-1}$ is called bi-graded spherical harmonic.
\smallskip

The following weighted functional relations can be obtained by considering
the Hecke-Bochner identity for the spectral projection of compactly supported
functions. For more details, see \cite{T2}, p. 98.

\begin{lemma}\label{lemma3C}\cite{T2}
For $z\in\mathbb C^n,$ let $P\in H_{p,q}$ and $d\nu_r=Pd\sigma_r,$ where $\sigma_r$ is
the surface measure on the sphere $S_r.$ Then
\[\varphi_k^{n-1}\times\nu_r(z)=
(2\pi)^{-n}\frac{\Gamma(k-q+1)}{\Gamma(k+n+p)}r^{2(p+q)}\varphi_{k-q}^{n+p+q-1}(r)P(z)\varphi_{k-q}^{n+p+q-1}(z),\]
if $k\geq q$ and ~$0$ otherwise.
\end{lemma}

We need the following basic facts about the bigraded
spherical harmonics, (see \cite{D, Gr1, T2} for details). Let $K=U(n)$
be the unitary group and $M=U(n-1).$ Then, $S^{2n-1}\cong K/M$ under
the map $kM\rightarrow k.e_n,$ where $k\in U(n)$ and $e_n=(0,\ldots
,1)\in \mathbb C^n.$ Let $\hat{K}_M$ denote the set of all
equivalence classes of irreducible unitary representations of $K$
which have a nonzero $M$-fixed vector.

\smallskip

For a $\tau\in\hat{K}_M,$ which is realized on $V_{\tau},$ let
$\{e_1,\ldots, e_{d(\tau)}\}$ be an orthonormal basis of
$V_{\tau}$ with $e_1$ as the $M$-fixed vector. Let
$t_{ij}^{\tau}(k)=\langle e_i,\tau (k)e_j \rangle ,$ $k\in K.$
By the Peter-Weyl theorem, $\{\sqrt{d(\tau
)}t_{j1}^{\tau}:1\leq j\leq d(\tau ),\tau\in\hat{K}_M\}$ forms an
orthonormal basis for $L^2(K/M),$ (see \cite{T2}).
Define $Y_j^{\tau} (\omega )=\sqrt{d(\tau )}t_{j1}^{\tau}(k),$
where $\omega =k.e_n\in S^{2n-1},$ and $k \in K.$ Then
$\{Y_j^{\tau}:1\leq j\leq d(\tau ),\tau\in \hat{K}_M\}$
becomes an orthonormal basis for $L^2(S^{2n-1}).$
\smallskip

Since $H_{p,q}$ is $K$-invariant, let $\pi_{p,q}$ be the
corresponding representation of $K$ on $H_{p,q}.$ Then $\hat{K}_M$
can be identified with $\{\pi_{p,q}: p,q\in\mathbb Z_+\}.$
See \cite{Ru}, p.253, for more details. Thus, a bi-graded spherical
harmonic on $S^{2n-1}$ can be defined by $Y_j^{p,q}(\omega)=\sqrt{d(p,q )}t_{j1}^{p,q}(k),$
and hence $\{Y_j^{p,q}:1\leq j\leq d(p,q),p,q \in \mathbb Z_+ \}$ forms
an orthonormal basis for $L^2(S^{2n-1}).$ For $f\in L^2(S^{2n-1}),$ the
expression
\begin{equation}\label{exp32}
\Pi_{p,q}(f)(\omega)=\sum_{j=1}^{d(p,q)}a_j^{p,q}Y_j^{p,q}(\omega),
\end{equation}
where $a_j^{p,q}\in\mathbb C,$ is called $(p,q)^{th}$ spherical
harmonic projection of $f.$

\smallskip

For each $l\in\mathbb Z_+,$ the space $H_l$ consists of spherical harmonics
of degree $l$ in $\mathbb{R}^d,$ is $SO(d)$ - invariant. When $d=2n,$ $H_l$
will be $U(n)$ - invariant as well, and under this action of $U(n),$ the
space $H_l$ breaks up into an orthogonal direct sum of $H_{p,q}$'s, where $p+q=l.$

\begin{lemma}\label{lemma4}\cite{Ru}.
Let $\omega\in S^{2n-1}$ and $Y\in H_l.$ Then there exists $Y_{p,q}\in H_{p,q},~p+q=l$ such that
\begin{equation}\label{exp33}
Y(\omega)=\sum_{p+q=l}Y_{p,q}(\omega).
\end{equation}
\end{lemma}

\begin{definition}\label{def1}
A set $\mathcal C\subset\mathbb C^n~(n\geq 2)$ that satisfies the scaling condition
$\lambda \mathcal C\subseteq\mathcal C$ for all $\lambda\in\mathbb C,$ is called a complex cone,
whereas a set $\mathcal D$ in $\mathbb R^d~(d\geq2)$ which satisfies $\lambda\mathcal D\subseteq\mathcal D,$
for all $\lambda\in\mathbb R$ is called a real cone.
\end{definition}

We say a cone is {\em non-harmonic} if it is not contained in the zero set
of any homogeneous harmonic polynomial. An example of a non-harmonic complex
cone was produced by the  author (see \cite{Sri4}). The zero set of the polynomial
$H(z)=az_1\bar z_2+|z|^2,$ where $a\not=0$ and $z\in\mathbb C^n,$ is a complex cone
which is not contained in the zero set of any bi-graded homogeneous harmonic polynomial.

\smallskip

An example of a non-harmonic real cone was given by Armitage, (see \cite{A}). Let $0<a<1.$
Then $K_a=\left\{x\in\mathbb R^d:~|x_1|^2=a^2|x|^2\right\}$ is a non-harmonic cone
if and only if $D^m G_k^{\frac{d-2}{2}}(a)\neq0,$ for all $0\leq m\leq k-2,$ where $D^m$
stands for the $m$th derivative on $\mathbb  R.$

\smallskip

In view of Lemma \ref{lemma4}, it is easy to prove the following result,
which is required to prove our main result.

\begin{lemma}\label{lemma5}
Let $Y\in H_l$ be given as in (\ref{exp33}). Suppose $\mathcal C$ be a complex cone, and
denote $\widetilde{\mathcal C}=\left\{\frac{z}{|z|}:~z\in\mathcal C, ~ z\neq0\right\}.$
Then $Y=0$ on $\widetilde{\mathcal C}$ if and only if  $Y_{p,q}=0$ on $\widetilde{\mathcal C},~\forall~
p, q\in\mathbb Z_+$ which are lying on the diagonal $p+q=l.$
\end{lemma}

\begin{proof}
Let $\omega\in\widetilde{\mathcal C}$ and $Y(\omega)=0.$  Since the cone $\mathcal C$ is closed
under complex scaling, by replacing $\omega$ with $e^{i\theta}\omega$ in (\ref{exp33}) we get
\[\sum_{p+q=l}e^{i(p-q)\theta}Y_{p,q}(\omega)=0.\]
Thus, the proof of the required lemma will be followed by the fact that $\{e^{is\theta}: ~s\in\mathbb Z\}$
is an orthogonal set in $L^2(S^1).$
\end{proof}

For each fixed $\xi\in S^{2n-1},$ define a linear functional on $H_l$
by $Y\mapsto Y(\xi).$ Then there exists a unique spherical harmonic,
say $Z_\xi^{(l)}\in H_l$ such that
\begin{equation}\label{exp1}
Y(\xi)=\int_{S^{d-1}}Z_\xi^{(l)}(\eta)Y(\eta)d\sigma(\eta).
\end{equation}
The spherical harmonic $Z_\xi^{(l)}$ is a $SO(2n)$ bi-invariant real-valued function,
which is constant on each geodesic, orthogonal to the line joining the origin and
$\xi.$ The spherical harmonic $Z_\xi^{(l)}$ is called the zonal harmonic of the
space $H_l$ at the pole $\xi.$ For more details, see \cite{SW}, p. 143.

\smallskip

For $1\leq t\leq\infty,$ let $f\in L^t(S^{2n-1}).$ For each $l\in\mathbb Z_+,$
we define the $l^{th}$ spherical harmonic projection of $f$ by
\begin{equation}\label{exp11}
\Pi_lf(\xi)=\int_{S^{2n-1}}Z_\xi^{(l)}(\eta)f(\eta)d\sigma(\eta).
\end{equation}
Then  $\Pi_lf$ is a spherical harmonic of degree $l.$ If for a $\delta>n-1,$
we denote $A_l^m(\delta)=\binom{m-l+\delta}{\delta}{\binom{m+\delta}{\delta}}^{-1},$
then the spherical harmonic expansion $\sum\limits_{l=0}^\infty\Pi_lf$ will be $\delta$-Cesaro
summable to $f.$ That is,
\begin{equation}\label{exp2}
f=\lim\limits_{m\rightarrow\infty}\sum_{l=0}^m~A_l^m(\delta)\Pi_lf,
\end{equation}
where limit in the right-hand side of (\ref{exp2}) exists in $L^t\left(S^{2n-1}\right).$
Further, the convergence in (\ref{exp2}) can be extended to hold in $L^t\left(rS^{2n-1}\right)$ while $r>0.$
For more details, see \cite{So}.

\smallskip

We would like to mention that the proof of our main result will be carried out
by concentrating the cone to the unit sphere and decomposing the integral
on sphere into averages over geodesic spheres. This is possible because the cone
is closed under scaling.

\smallskip
For $\omega\in S^{2n-1}$ and $t\in(-1, 1),$ the set $S_\omega^t=\left\{\nu\in S^{2n-1}: \omega\cdot\nu=t\right\}$
is a geodesic sphere on $S^{2n-1}$ with a pole at $\omega.$ Let $f$ be an integrable function
on $S^{2n-1}.$ Then in view of Fubini's Theorem, we can define the geodesic spherical means
of $f$ by
\[\tilde f(\omega, t)=\int_{S_\omega^t}f d\sigma_{2n-2},\]
where $\sigma_{2n-2}$ stands for the normalized surface measure on the geodesic sphere $S_\omega^t.$

\smallskip

Since the zonal harmonic $Z_\xi^{(l)}$ is $SO(2n)$ bi-invariant, three exists
a nice function $F$ on $(-1,1)$ satisfying $Z_\xi^{(l)}(\eta)=F(\xi\cdot\eta).$
Hence the extension of the formula (\ref{exp1}) for the functions $F\in L^1(-1,1)$
is inevitable. The latter fact is known as Funk-Hecke formula. In other words,
\begin{equation}\label{exp03}
\int_{S^{2n-1}}F(\xi\cdot\eta)Y(\eta)d\sigma(\eta)= C_lY(\xi),
\end{equation}
where the constant $C_l$ is given by
\[C_l=\alpha_l\int_{-1}^1F(t)G_l^{n-1}(t)(1-t^2)^{\frac{2n-3}{2}}~dt,\]
and $G_l^\beta$ stands for the Gegenbauer polynomial of degree $l$ and order $\beta.$
For more details, see
\cite{AAR}, p. 459. We shall mention the following lemma which percolates the geodesic
mean vanishing conditions of $f\in L^1(S^{2n-1})$ to vanishing condition of
each spherical harmonic projection of $f$. For the class of continuous functions
$C(S^{2n-1}),$ this lemma was proved in \cite{AVZ}. In \cite{Sri5}, the  author
has extended the result for functions in $L^1(S^{2n-1})$ with the help of the
Cesaro summation formula (\ref{exp2}).

\begin{lemma}\label{lemma3}\cite{Sri5}
Let  $f\in L^1(S^{2n-1}).$ Then $\tilde f(\omega, t)=0~$ for all $t\in (-1, 1)$
if and only if $\Pi_lf(\omega)=0$ for all $l\in\mathbb Z_+.$
\end{lemma}
Notice that, as a corollary to Lemma \ref{lemma3}, it can be shown that if
$\tilde f(\omega, t)=0$ for all $(\omega,t)\in\widetilde{\mathcal C}\times{(-1,1)},$
then $f=0$ on $S^{2n-1}$ as long as $\mathcal C$ is happened to be non-harmonic.

\section{Uniqueness pairs for the symplectic Fourier transform}\label{section3}
In this section, we prove that the unit sphere $S^{2n-1}$ together with a non-harmonic
complex cone forms a Heisenberg uniqueness pair for the symplectic Fourier transform (SFT).

\smallskip

Let $X(S^{2n-1})$ denote the space of all finite Borel measures $\mu$ in $\mathbb C^n$
which are supported on $S^{2n-1}$ and absolutely continuous with respect to the
surface measure on $S^{2n-1}.$ By Radon-Nikodym theorem, there exists
$f\in L^1(S^{2n-1})$ such that $d\mu=fd\sigma.$ Define the symplectic Fourier
transform of a measure $\mu\in X(S^{2n-1})$ by
\[\mathcal F_S\mu(z)=\int_{S^{2n-1}}e^{-\frac{i}{2}\text{Im}(z\cdot\bar\zeta)}f(\zeta)d\sigma(\zeta),\]
where $z=x+iy\in\mathbb C^n$ and $\zeta=\xi+i\eta\in\mathbb C^n$. Hence
$\mathcal F_S\mu$ is a bounded uniformly continuous function on $\mathbb C^n.$
In other words, $\mathcal F_S\mu$ can be expressed as
\begin{equation}\label{exp9}
\mathcal F_S\mu(x,y)=\int_{S^{2n-1}}e^{-\frac{i}{2}(-x\cdot\eta+y\cdot\xi)}f(\xi,\eta)d\sigma(\xi,\eta).
\end{equation}
We are intended to prove the following result.
\begin{theorem}\label{th1}
Let $\mathcal C$ be a complex cone in $\mathbb C^n$  and $\mu\in X(S^{2n-1}).$ If
$\mathcal{F}_S\mu(z)=0$ for all $z\in\mathcal C,$ then $\mu=0$ if and only if
$\mathcal C$ is non-harmonic.
\end{theorem}
\begin{proof}
Let $(x,y)=r\omega,$ where $\omega=(\omega_1,\ldots,\omega_n,\omega'_1,\ldots,\omega'_n)\in S^{2n-1}.$
Denote $\tilde \omega=(\omega'_1,\ldots,\omega'_n,-\omega_1,\ldots,-\omega_n).$ Then from (\ref{exp9})
it implies that
\begin{equation}\label{exp3}
\int_{S^{2n-1}}e^{-\frac{i}{2}r\tilde\omega\cdot (\xi,\eta)}f(\xi,\eta)d\sigma(\xi,\eta)=0,
\end{equation}
whenever $r\omega\in\mathcal C.$ Since $\mathcal C$ is closed under complex scaling, $r\omega\in\mathcal C$
implies $r\tilde \omega\in\mathcal C.$ By decomposing the integral in (\ref{exp3}) over geodesic spheres at
pole $\omega,$ we obtain
\[\int_{-1}^1\left(\int_{S_\omega^t} e^{-\frac{i}{2}{r\omega\cdot\nu}}f(\nu)d\sigma_{2n-2}(\nu)\right)dt=0,\]
where $S_\omega^t=\left\{\nu\in S^{2n-1}: \omega\cdot\nu=t\right\}.$  That is,
\begin{equation}\label{exp7}
\int_{-1}^1e^{-\frac{i}{2}rt}\tilde f(\omega,t)dt=0,
\end{equation}
for all $r>0.$ Hence $\tilde f(\omega,t)=0,$ for all $t\in(-1,1).$ Thus, by
Lemma \ref{lemma3}, it follows that $\Pi_l(f)(\omega)=0$ for all $l\in\mathbb Z_+.$
Further, by Lemma \ref{lemma5}, we get $\Pi_{p,q}f(\omega)=0$ for all $p,q\in\mathbb Z_+.$
Thus, by the given condition that $w\notin Y_{pq}^{-1}(0)$ for any $p,q\in\mathbb Z_+,$
it follows that $f=0.$ That is, $\mu=0.$

Conversely, assume that $\mathcal C$ is contained in the zero set of a bigraded
homogeneous harmonic polynomial $P\in H_{p,q}$ and denote $Y=P|_{S^{2n-1}}.$
For $\zeta\in S^{2n-1},$ let $d\mu(\zeta)=Y(\zeta)d\sigma(\zeta).$ Then $\mu$
is a finite complex Borel measure supported on $S^{2n-1}.$ By identifying $\mathbb C^n$
with $\mathbb R^{2n},$ we find that $H_{p,q}\subseteq H_{p+q}.$ Therefore, for
 $z=r\omega\in\mathcal C,$ we can summarize that
\begin{align}\label{exp17}
\mathcal F_S\mu(z)=\int_{S^{2n-1}} e^{-ir\tilde \omega\cdot \zeta} Y_{p+q}(\zeta) d\sigma(\zeta)
=(2\pi)^n i^{p+q}\dfrac{J_{p+q+n-1}(r)}{r^{n-1}} Y_{p+q}(\tilde \omega),
\end{align}
where the last identity was obtain in \cite{AAR}, as a corollary of the Funk-Hecke formula (\ref{exp03}).
Recall that $r\omega\in \mathcal{C}$ implies $r\tilde \omega\in\mathcal C.$ Thus, from (\ref{exp17}),
we can conclude that $\mathcal F_S \mu|_\mathcal C=0.$
\end{proof}

\begin{remark}\label{rk1} $(a)$ Further, we observed that Theorem \ref{th1}  holds for
a {\em non-harmonic real cone}. Let $\mathcal C$ be a non-harmonic real cone. Write
$\tilde\omega=\sigma_o\omega,$ where $\sigma_o$ is the symplectic matrix, which in fact,
belongs to $U(n)\subset SO(2n).$ Suppose $\mu\in X(S^{2n-1})$ satisfies $\mathcal F_S\mu\vert_\mathcal C=0.$
Then $\Pi_lf(\sigma_o\omega)=0$ for all $l\in\mathbb Z_+.$ Since $\sigma_o^{-1}\cdot\Pi_lf$
would also be a spherical harmonic, we infer that $(S^{2n-1},\mathcal C)$ is a HUP for SFT.

\smallskip

$(b)$ Suppose $\Gamma$ be a smooth sub-manifold in $\mathbb{R}^{2n}$ and $\Lambda$
be a subset of $\mathbb{R}^{2n}.$ Let $T:\mathbb{R}^{2n}\rightarrow\mathbb{R}^{2n}$
be defined by $T(x,y)=(\frac{y}{2},-\frac{x}{2})$ for $x,y\in\mathbb{R}^n.$
It is easy to see that $\mathcal{F}_S\mu(x,y)=\hat{\mu}(\frac{y}{2},-\frac{x}{2}),$
where $\hat{\mu}$ is the Euclidean Fourier transform (EFT) of $\mu.$ Thus, $(\Gamma,\Lambda)$
is a HUP for SFT if and only if $(\Gamma,T\Lambda)$ is a HUP
for the EFT.
\smallskip

For instance, by the Euclidean result (\cite{V1}, Proposition $1.2$),
$(S^{2n-1},S_r^{2n-1})$ is a HUP for the SFT as long as $\frac{r}{2}\notin J_{(n+k-1)}^{-1}(0)$
for any $k\in\mathbb Z_+.$

$(c)$ Consider the map $\tilde T:=(T^{-1})^*$ on $\mathbb{R}^{2n}.$ It is known for the
EFT that $(\Gamma,\Lambda)$ is a HUP if and only if $(\tilde T^{-1}\Gamma,\tilde T^*\Lambda)$
is a HUP. Hence, from Remark $(b),$ we have $(\Gamma,\Lambda)$ is a HUP for SFT if and only if
$(T^*\Gamma,\Lambda)$ is a HUP for EFT. Thus, in view of the Euclidean result (\cite{Sri5}, Theorem $3.1$),
Remark $(a)$ holds true.
\end{remark}

\section{Uniqueness pairs for the modified Fourier transform on $\mathbb H^n$}\label{section4}
In this section, we prove that a finite measure supported on the cylinder
$S_r^{2n-1}\times\mathbb R$ can be determined by a non-harmonic cone as
well as by the boundary of a bounded domain in $\mathbb C^n.$

\smallskip

We know from \cite{T1} that the modified Fourier transform of $f\in L^1(\mathbb H^n)$ is defined by
\[\mathcal{F}_M f(z,\lambda)=\pi_\lambda(-z)W_\lambda(f^\lambda)\pi_\lambda(z),\]
where $W_\lambda(f^\lambda)$ is the Weyl transform of $f^\lambda$ and
$(z,\lambda)\in\mathbb C^n\times\mathbb R^\ast.$ This, in turn, can be expressed as
\begin{align*}
\mathcal{F}_M f(z,\lambda) &=\int_{\mathbb C^n}\pi_\lambda(-z)\pi_\lambda(w)f^\lambda(w)\pi_\lambda(z)dw, \\
&=\int_{\mathbb C^n}e^{-i\lambda\text{Im}(z\cdot\bar w)}f^\lambda(w)\pi_\lambda(w)dw.
\end{align*}
Consider the measure $\mu\in X(S_r^{2n-1}\times\mathbb R).$ Then there exists
$f\in L^1(S_r^{2n-1}\times\mathbb R)$ such that $d\mu(\zeta,t)=f(\zeta,t)d\sigma_r(\zeta) dt.$
For $(z,\lambda)\in\mathbb C^n\times\mathbb R^\ast,$ define the modified Fourier
transform of $\mu$ by
\[\mathcal{F}_M\mu(z,\lambda)=\int_{S_r^{2n-1}}e^{-i\lambda\text{Im}(z\cdot\bar\zeta)}f^\lambda(\zeta)
\pi_\lambda(\zeta)d\sigma_r(\zeta).\]
For $\lambda\in\mathbb{R}^*,$ consider the subspace
$U_\lambda=\text{ span}\{\phi_{\alpha_\lambda}^\lambda\}$ of $L^2(\mathbb{R}^n),$ where
$\phi_{\alpha_\lambda}^\lambda$ is the scaled Hermite function for some $\alpha_\lambda\in\mathbb{Z}_+^n.$
Denote $\tilde \Lambda=\mathcal C\times\mathbb R^\ast,$ where $\mathcal C$ could be a real (or complex) cone.
We have proved the following result.
\begin{proposition}
Let $\mu\in X(S_r^{2n-1}\times\mathbb R)$ and range of $\mathcal{F}_M\mu(\xi,\lambda)$ is a subspace of
$U_\lambda^\perp$ for all $(\xi,\lambda)\in\tilde \Lambda.$ If $\mathcal C$ is non-harmonic, then $\mu=0.$
\end{proposition}

\begin{proof}
Since, for each $\lambda\in \mathbb R^*,$ range of $\mathcal{F}_M\mu(\xi,\lambda)$
is a subspace of $U_\lambda^\perp,$ it follows that
\begin{equation}\label{exp15}
\langle\mathcal{F}_M\mu(z,\lambda)\varphi,\phi_{\alpha_\lambda}^\lambda\rangle=0
\end{equation}
for all $\varphi\in L^2(\mathbb R^n).$ We know from \cite{T1} that
$\langle\pi_\lambda(w)\phi_o^\lambda,\phi_\alpha^\lambda\rangle
=c_\alpha |\lambda|^{|\alpha|/2}w^\alpha e^{\frac{-|\lambda||w|^2}{4}}$
for all $w\in\mathbb{C}^n.$ If we choose $\varphi=\phi_o^\lambda,$ then from (\ref{exp15})
we have
\begin{align}\label{exp19}
\int_{S_r^{2n-1}}e^{-i\lambda\text{Im}(z\cdot\bar\zeta)}f^\lambda(\zeta)
c_{\alpha_\lambda}\zeta^{\alpha_\lambda} e^{\frac{-|\lambda||\zeta|^2}{4}}d\sigma_r(\zeta)=0
\end{align}
for all $z\in\mathcal C.$ This, in fact, reduces to the case of SFT
on $\mathbb C^n.$  That is, for each $\lambda\in\mathbb R^*,$ we get $\mathcal F_S(g_r^\lambda)(2r\lambda z)=0$
for all $z\in \mathcal C,$ where $g_r^\lambda(\nu)=f^\lambda(r\nu)\nu^{\alpha_\lambda}$
for $\nu\in S^{2n-1}.$ Since $(2r\lambda) \mathcal C\subseteq \mathcal C,$ in view of Theorem
\ref{th1} and Remark \ref{rk1} $(a),$ we infer that $f^\lambda=0$ if and only if $\mathcal C$
is non-harmonic. Thus, $f=0.$
\end{proof}

\begin{remark}
Let $\mu\in X(\Gamma\times\mathbb R)$ and range of $\mathcal{F}_M\mu(z,\lambda)$ is a subspace of
$U_\lambda^\perp,$ whenever $(z,\lambda)\in\Lambda\times\mathbb R^*.$ If $(\Gamma,sT\Lambda)$ is a HUP
for EFT for almost all $s\in \mathbb R,$ then Remark \ref{rk1}
$(b)$ and Equation (\ref{exp19}) allow to conclude that $\mu=0.$
\smallskip

For instance, consider $\Gamma=S^{2n-1}$
and $\Lambda=S_r^{2n-1}.$ Since the set $\{J_{n+k-1}^{-1}(0):k\in\mathbb Z_+\}$ has measure zero,
in view of the Euclidean result (\cite{V1}, Proposition $1.2$), we conclude $\mu=0.$
\end{remark}

\begin{theorem}\label{th9}
Let $\partial\Omega$ be the boundary of the bounded domain $\Omega$ in $\mathbb C^n.$
If $\mu\in X(S_r^{2n-1}\times\mathbb R)$ satisfies $\mathcal F_M\mu(\xi,\lambda)=0$ for all
$(\xi,\lambda)\in\partial\Omega\times\mathbb R^\ast,$ then $\mu=0.$
\end{theorem}

\begin{proof}
Since $\mathcal F_M\mu$ can be extended holomorphically to a function $F(.,\lambda)$ on $\mathbb C^{2n},$
taking values in $L^2(\mathbb R^n),$ it follows that $F(.,\lambda)|_{\mathbb R^{2n}}=\mathcal F_M\mu$ is
a real analytic function. Consider
\[\mathcal F_M\mu(z,\lambda)=\int_{S_r^{2n-1}}e^{-i\lambda\text{Im}(z\cdot\bar\zeta)}
f^\lambda(\zeta)\pi_\lambda(\zeta)d\sigma_r(\zeta).\]
Then \[\frac{\partial}{\partial z_j}\mathcal F_M\mu(z,\lambda)=\frac{\lambda}{2}
\int_{S_r^{2n-1}}\bar{\zeta_j}e^{-i\lambda\text{Im}(z\cdot\bar\zeta)}f^\lambda(\zeta)\pi_\lambda(\zeta)d\sigma_r(\zeta)\]
and \[\frac{\partial^2}{\partial\bar{z_j}\partial z_j}\mathcal F_M\mu(z,\lambda)=-\frac{\lambda^2}{4}\int_{S_r^{2n-1}}
\bar{\zeta_j}\zeta_je^{-i\lambda\text{Im}(z\cdot\bar\zeta)}f^\lambda(\zeta)\pi_\lambda(\zeta)d\sigma_r(\zeta).\]
This, in turn, implies
\[\triangle_z\mathcal F_M\mu(z,\lambda)+(r\lambda)^2\mathcal F_M\mu(z,\lambda)=0.\]
Now, if $\varphi,\psi\in L^2(\mathbb R^n),$ then we have
\[\triangle_z\langle\mathcal F_M\mu(z,\lambda)\varphi,\psi\rangle
+(r\lambda)^2\langle\mathcal F_M\mu(z,\lambda)\varphi,\psi\rangle=0.\]
Let $g(z,\lambda)=\langle\mathcal F_M\mu(z,\lambda)\varphi,\psi\rangle.$ Then $g$
will be a real analytic function satisfying
\[\triangle_\xi g(z,\lambda)+(r\lambda)^2g(z,\lambda)=0.\]
Hence $g(., \lambda);~\lambda\in\mathbb R^\ast$ are eigenfunctions of the Dirichlet
boundary value problem in $\mathbb C^n.$ By the discreetness of eigenvalues of
the Dirichlet problem in the bounded domain, it follows that $g(.,\lambda)=0$
for all most all $\lambda\in\mathbb R^\ast.$ Since $g(.,\lambda)$ is continuous in $\lambda,$
we infer that $g(z,\lambda)=0$ for all $(z,\lambda)\in\mathbb C^n\times\mathbb R^\ast.$
Thus, $\mu=0.$
\end{proof}

\begin{remark}
Let $u$ be a solution of the Helmholtz equation $\triangle u+c^2u=0$
on $\mathbb R^{2n}$ and $\Sigma\subset\mathbb R^{2n}$ be such that
$u=0$ on $\Sigma$ implies $u=0$ a.e. Then the statment of Theorem \ref{th9}
will remain true if we replace $\partial\Omega$ by $\Sigma.$
Regular Jordan curves and two intersecting curves separated by
angle which is an irrational multiple of $\pi,$ are examples of
such $\Sigma.$ For details, see \cite{FGJ}.
\end{remark}

\section{Uniqueness pairs for the spectral projections}\label{section5}

In this section, we derive that sphere is determining set for the spectral
projections of finite measures on $\mathbb C^n,$ which are supported on $S^{2n-1}.$
Further, we deduce that non-harmonic complex cone as well as $NA$-set can determine
the spectral projections of the above class of measures.

\smallskip

For $\mu\in X(S_r^{2n-1}),$ we define the spectral projection of $\mu$ by
\[\varphi_k^{n-1}\times\mu(z)=\int_{S_r^{2n-1}}\varphi_k^{n-1}(z-\omega)
e^{\frac{i}{2}\text{Im}(z\cdot\bar{\omega})}d\mu(\omega).\]
Then the following result holds true.
\begin{theorem}\label{th2}
Let $\mu\in X(S_{r_1}^{2n-1})$ be such that  $\varphi_k^{n-1}\times\mu(z)=0$ for all $z\in S_{r_2}^{2n-1}$
and $k\in\mathbb Z_+.$ Then $\mu=0.$
\end{theorem}
In order to prove Theorem \ref{th2}, we need the following results about
irreducibility of the Laguerre polynomials.
\begin{theorem}\cite{FL}\label{th4}
Let $\upsilon$ be a rational number, which is not a negative integer. Then for all
but finitely many $k\in\mathbb{Z}_+,$ the Laguerre polynomial $L_k^\upsilon$ is
irreducible over the field of rationals.
\end{theorem}
Notice that the disjointness of the zero set of Laguerre functions can
be understood by the following description of the zero set of Laguerre
polynomials. This follows from Theorem \ref{th4}, and has worked out in \cite{Sri3}.
\begin{proposition}\cite{Sri3}\label{prop5}
Let $k\in\mathbb{Z}_+.$ Then for all but finitely many $k,$ the Laguerre polynomials $L_k^{n-1}$ have
distinct zeroes over the reals.
\end{proposition}

\smallskip

\begin{proof}[\textbf{Proof of Theorem \ref{th2}}]
Since $\mu\in X(S_{r_1}^{2n-1}),$ there exists $f\in L^1(S_{r_1}^{2n-1})$ such that $d\mu=fd\sigma_{r_1}.$ Thus,
\begin{equation}\label{exp8}
\varphi_k^{n-1}\times\mu(z)=\int_{S_{r_1}^{2n-1}}\varphi_k^{n-1}(z-\omega)
e^{\frac{i}{2}\text{Im}(z\cdot\bar{\omega})}f(\omega)d\sigma_{r_1}(\omega)=0,
\end{equation}
for all $k\in\mathbb Z_+$ and $z\in S_{r_2}^{2n-1}.$ As $f\in L^1(S_{r_1}^{2n-1}),$ $f$ will satisfy
\[f=\lim\limits_{m\rightarrow\infty}\sum_{l=0}^m~A_l^m(\delta)\Pi_lf,\]
where $A_l^m(\delta)=\binom{m-l+\delta}{\delta}{\binom{m+\delta}{\delta}}^{-1}$ and $\delta>n-1.$\
Further, from Lemma \ref{lemma4}, it follows that
\[f=\lim\limits_{m\rightarrow\infty}\sum_{l=0}^m\sum_{p+q=l}~A_{p+q}^m(\delta)\Pi_{p,q}f.\]
Now, from condition (\ref{exp8}), it follows that
{\footnotesize \begin{align*}
\left|\sum_{l=0}^m\sum_{p+q=l}A_{p+q}^m(\delta)\varphi_k^{n-1}\times\Pi_{p,q}f(z)\right|=
\left|\sum_{l=0}^m\sum_{p+q=l}A_{p+q}^m(\delta)\varphi_k^{n-1}\times\Pi_{p,q}f(z)-\varphi_k^{n-1}\times f(z)\right| \\
\leq M_k\int_{S_{r_1}^{2n-1}}\left|\sum_{l=0}^m\sum_{p+q=l}A_{p+q}^m(\delta)
\Pi_{p,q}f(\omega)-f(\omega)\right|d\sigma_{r_1}(\omega),
\end{align*}}\\
where $M_k=\sup_{(\omega,z)\in S_{r_1}^{2n-1}\times S_{r_2}^{2n-1}}|\varphi_k^{n-1}(z-\omega)|.$
Hence in view of (\ref{exp2}), we deduce that
\begin{equation}\label{exp5}
\lim\limits_{m\rightarrow\infty}\sum_{l=0}^m\sum_{p+q=l}~A_{p+q}^m(\delta)\varphi_k^{n-1}\times\Pi_{p,q}f=0
\end{equation}
converges uniformly in $S_{r_2}^{2n-1},$ whenever $k\in\mathbb Z_+.$ When $k\geq q,$
Lemma \ref{lemma3C} gives
\begin{equation}\label{exp12}
\int_{S^{2n-1}}\varphi_k^{n-1}(z-r_1\eta)e^{\frac{i}{2}r_1\text{Im}(z\cdot\bar{\eta})}Y_{p,q}(\eta)d\sigma(\eta)
=B_n^{k,\gamma}{r_1}^{p+q}\varphi_{k-q}^{\gamma-1}(r_1)\varphi_{k-q}^{\gamma-1}(z)P_{p,q}(z),
\end{equation}
where $B_n^{k,\gamma}=(2\pi)^{-n}\dfrac{\Gamma(k-q+1)}{\Gamma(k+n+p)}$ and $\gamma=n+p+q.$
Let $z=r_2\xi,$ where $\xi\in S^{2n-1}.$ Then from (\ref{exp5}) and (\ref{exp12}), we infer that
\begin{equation}\label{exp10}
\lim\limits_{m\rightarrow\infty}\sum_{l=0}^m\sum_{p+q=l}A_{p+q}^m(\delta)B_n^{k,\gamma}{(r_1r_2)}^{p+q}
\varphi_{k-q}^{\gamma-1}(r_1)\varphi_{k-q}^{\gamma-1}(r_2)\Pi_{p,q}f(\xi)=0.
\end{equation}
Since (\ref{exp5}) converges uniformly on $S_{r_2}^{2n-1},$ it follows that (\ref{exp10})
converges in $L^2(S^{2n-1}).$ Recall that the bi-gradted spherical harmonic projections $\Pi_{p,q}f$
are orthogonal among themselves, and $A_{p+q}^m(\delta)B_n^{k,\gamma}\neq 0$ holds true for every choice of
$p,q\in\mathbb Z_+.$ From (\ref{exp10}) we obtain that
\begin{equation}\label{exp14}
\varphi_{k-q}^{\gamma-1}(r_1)\varphi_{k-q}^{\gamma-1}(r_2)\left\|\Pi_{p,q}f\right\|_2=0,
\end{equation}
whenever $k\geq q.$ Hence by invoking Proposition \ref{prop5} for each pair of $p,q\in\mathbb{Z}_+,$
there exists $k_o\geq q$ such that $r_i\notin\left(\varphi_{k_o-q}^{n+p+q-1}\right)^{-1}(0);$
when $i=1,2.$ Hence, from (\ref{exp14}) we conclude that $\Pi_{p,q}f=0$ for all
$p,q\in\mathbb Z_+.$ Thus, $f=0.$
\end{proof}

\begin{remark}\label{rk2}
 A set, which is a determining set for any real analytic function, is called
$NA$ - set. For instance, the spiral is an $NA$ - set in the plane
(see \cite{PS_0}). Since the spectral projection $\varphi^{n-1}_k\times\mu$
can be extended holomorphically to $\mathbb C^{2n},$ the function
$\varphi^{n-1}_k\times\mu$ must be real analytic on $\mathbb C^n.$

\smallskip

Let $\Lambda$ be an NA-set for real analytic functions on $\mathbb C^n.$
If $\mu\in X(S_r^{2n-1})$ satisfies $\varphi^{n-1}_k\times\mu\vert_{\Lambda}=0$
for all $k\in\mathbb Z_+,$ then $\varphi^{n-1}_k\times\mu(z)=0$
for all $z\in\mathbb C^n.$ Thus, in view of Theorem \ref{th2}, we have $f=0.$
\end{remark}

Next, we shall prove that spectral projections of a measure $\mu\in X(S_r^{2n-1})$
can be determined by a non-harmonic complex cone.

\begin{theorem}\label{th3}
Let $\mathcal{C}$ be a non-harmonic complex cone in $\mathbb C^n.$ If $\mu\in X(S_r^{2n-1})$
satisfies $\varphi^{n-1}_k\times\mu\vert_{\mathcal{C}}=0$ for all $k\in\mathbb Z_+,$ then $\mu=0.$
\end{theorem}
\begin{proof}
From (\ref{exp10}), it follows that
\[\lim\limits_{m\rightarrow\infty}\sum_{l=0}^m\sum_{p+q=l}A_{p+q}^m(\delta)B_n^{k,\gamma}{(rs)}^{p+q}
\varphi_{k-q}^{\gamma-1}(r)\varphi_{k-q}^{\gamma-1}(s)\Pi_{p,q}f(\xi)=0,\]
whenever $s\xi\in \mathcal C$ and $k\in\mathbb Z_+.$
Since $\mathcal{C}$ is closed under complex scaling, replacing $\xi$ by $e^{i\theta}\xi,$ we obtain
\[\lim\limits_{m\rightarrow\infty}\sum_{l=0}^m\sum_{p+q=l}A_{p+q}^m(\delta)B_n^{k,\gamma}{(rs)}^{p+q}
\varphi_{k-q}^{\gamma-1}(r)\varphi_{k-q}^{\gamma-1}(s)\Pi_{p,q}f(\xi)e^{i(p-q)\theta}=0.\]
Now, by induction on $k,$ we show that each of the projection $\Pi_{p,q}f$ restricted to
$\mathcal{C}$ must be zero. For $k=0,$ the choice for $q$ is only $0.$ Since
$\{e^{ip\theta}: p\in\mathbb Z_+\}$ is an orthogonal set, we infer that $\Pi_{p,0}(f)(\xi)=0.$
Similarly, for $k=1,$ the choices for $q$ are $0$ and $1$ only. The case $q=0$ is already settled. Now
for $q=1,$ $\{e^{i(p-1)\theta}: p\in\mathbb Z_+\}$ is an orthogonal set. Hence
$\Pi_{p,1}(f)(\xi)=0.$ This, in turn, implies that each of $\Pi_{p,q}(f)$
vanishes on $\mathcal{C}.$ Thus $f=0.$
\end{proof}

\section{Benedicks-Amrein-Berthier theorem}\label{section}
In this section, we prove Benedicks-Amrein-Berthier  theorem for
the Heisenberg group.

For $\lambda=1,$ we denote $W_1(g)$ by $W(g).$ For $g\in L^2(\mathbb{C}^n),$
suppose $W(g)$ is of finite rank. Then there exists an orthonormal basis
$\{e_1,e_2,\ldots\}$ of $L^2(\mathbb{R}^n)$ such that $\mathcal{R}(W(g))=\mathcal{B}_N,$
where $\mathcal{B}_N=\text{span}\{e_1,\ldots,e_N \}$ and $\mathcal R$ stands for
the range. Define an orthogonal projection $P_N $ of $L^2(\mathbb{R}^n)$ onto
$\mathcal{B}_N.$ Let $A$ be a measurable subset of $\mathbb{C}^n.$ Define
a pair of orthogonal projections $E_A$ and $F_N$ of $L^2(\mathbb{C}^n)$ by
\begin{equation}\label{exp16}
E_A g=\chi_A g \qquad \text{ and } \qquad W(F_N g)=P_N W(g),
\end{equation}
where $\chi_A$
denotes the characteristic function of $A.$ Then
$\mathcal{R}(E_A)=\{g\in L^2(\mathbb{C}^n): g= \chi_A g \}$ and
$\mathcal{R}(F_N)=\{g\in L^2(\mathbb{C}^n): \mathcal{R}(W(g))\subseteq\mathcal{B}_N \}.$

\smallskip

First, we prove that $E_AF_N$ is a Hilbert-Schmidt operator that satisfies
$\|E_AF_N\|_{HS}^2=(2\pi)^{-n}m(A)N.$ Throughout this section, $A$  will be
considered as a set of finite Lebesgue measure.

\begin{lemma}\label{lemma2}
The operator $E_AF_N$ is an integral operator with kernel
$K(z,w)=(2\pi)^{-n}\chi_A(z)\text{tr}\left(P_N \pi(w)\pi(-z)\right),$ where $z,w\in\mathbb{C}^n.$
\end{lemma}
\begin{proof}
For $g\in L^2(\mathbb{C}^n)$, we have $W(F_N g)=P_N W(g)$. By inversion formula for the
Weyl transform, we have
\begin{align*}
(F_Ng)(z)&=(2\pi)^{-n}\text{tr}(\pi(z)^{*}W(F_Ng))=(2\pi)^{-n}\text{tr}(\pi(-z)P_NW(g))\\
&=(2\pi)^{-n}\text{tr}(P_NW(g) \pi(-z))\\
&=(2\pi)^{-n}\int_{\mathbb{C}^n} g(w) \text{tr}\left(P_N \pi(w)\pi(-z)\right)dw.
\end{align*}
Hence, we can write
\begin{align*}
(E_AF_Ng)(z)&=\chi_A(z)(F_Ng)(z)
=(2\pi)^{-n}\chi_A(z)\int_{\mathbb{C}^n} g(w)\text{tr} \left(P_N \pi(w)\pi(-z)\right)dw\\
&=\int_{\mathbb{C}^n} g(w)K(z,w)dw,
\end{align*}
where $K(z,w)=(2\pi)^{-n}\chi_A(z)\text{tr}\left(P_N \pi(w)\pi(-z)\right)$.
By this, we infer that $E_AF_N$ is an integral operator with kernel $K.$
\end{proof}

\begin{lemma}\label{lemma6}
$E_AF_N$ is Hilbert-Schmidt and $\|E_AF_N\|_{HS}^2=(2\pi)^{-n}m(A)N$.
\end{lemma}
\begin{proof}
From Lemma \ref{lemma2}, we know that  $E_AF_N$ is an integral operator with kernel $K(z,w).$
Therefore,
\begin{align*}
\|E_AF_N\|_{HS}^2&= \int_{\mathbb{C}^n}\int_{\mathbb{C}^n}|K(z,w)|^2 dw dz \\
&=(2\pi)^{-2n}\int_{\mathbb{C}^n}|\chi_A(z)|^2 \left(\int_{\mathbb{C}^n}
|\text{tr} \left(P_N \pi(w)\pi(-z)\right)|^2 dw\right)dz \\
&=(2\pi)^{-2n}\int_{\mathbb{C}^n}\chi_A(z) \left(\int_{\mathbb{C}^n}\Big|
\sum\limits_{j=1}^N \langle\pi(w)\pi(-z)e_j,e_j\rangle \Big|^2 dw\right)dz
\end{align*}
Since $\pi(w)\pi(z)=e^{\frac{i}{2}Im(w\cdot\bar{z})}\pi(w+z),$ we get
\begin{align*}
\|E_AF_N\|_{HS}^2&=(2\pi)^{-2n}\int_{\mathbb{C}^n}\chi_A(z) \int_{\mathbb{C}^n}\Big|
e^{-\frac{i}{2}Im(w\cdot\bar{z})}\sum\limits_{j=1}^N \langle\pi(w-z)e_j,e_j\rangle \Big|^2 dwdz \\
&=(2\pi)^{-2n}\int_{\mathbb{C}^n}\chi_A(z) \int_{\mathbb{C}^n}\Big|\sum\limits_{j=1}^N
\langle\pi(w)e_j,e_j\rangle \Big|^2 dwdz.
\end{align*}
Hence, from the orthogonality relation (\ref{exp28}) of the Fourier-Winger transform,
the lemma will be followed.
\end{proof}

We need the following result from \cite{AB} that describes an interesting
property of measurable sets having finite measure. Denote $w A=\{z\in \mathbb{C}^n:z-w\in A\}.$

\begin{lemma}\cite{AB}\label{lemma6.3}
Let $B$ be a measurable set in $\mathbb{C}^n$ with $0<m(B)<\infty.$ If $B_0$
is a measurable subset of $B$ with $m(B_0)>0,$ then for $\epsilon>0,$ there exists
$w\in \mathbb{C}^n$ such that \[m(B)<m(B\cup w B_0)<m(B)+\epsilon.\]
\end{lemma}

Now, for orthogonal projections $E$ and $F$ on a Hilbert space $\mathcal H,$
let $E\cap F$ denote the orthogonal projection of $\mathcal H$ onto
$\mathcal{R}(E)\cap \mathcal{R}(F).$ Then, we have the relation
\begin{align}\label{exp24}
\Vert E\cap F \Vert_{HS}^2=\dim\mathcal{R}( E\cap F)\leq \Vert EF \Vert_{HS}^2.
\end{align}
We abbreviate  $A^c=\mathbb{C}^n\smallsetminus A$ and $F_{N}^{\perp}=I-F_N.$
Let $S$ be a closed subspace of $L^2(\mathbb{R}^n).$ Define $F_S$ by
$W(F_Sg)=P_SW(g),$ where $P_S$ is the orthogonal projection of $L^2(\mathbb{R}^n)$
onto $S$ and $g\in L^2(\mathbb{C}^n).$ In particular, if $S=\mathcal{B}_N,$
then $F_S=F_N.$

\begin{proposition}\label{prop4}
Let $A \subset \mathbb{C}^n$ with finite measure, and $S$ be a closed subspace of $L^2(\mathbb{R}^n).$
Then, either $E_A\cap F_S=0$ or for every $\epsilon'>0,$ there exists $\tilde{A}\supset A$ with
$m(\tilde{A}\smallsetminus A)<\epsilon'$ such that $\mathcal{R}(E_{\tilde{A}}\cap F_S)$ is of
infinite dimensional.
\end{proposition}

\begin{proof}
If $E_A\cap F_S\neq 0,$ then there exists a non-zero function $g_0\in\mathcal{R}( E_A\cap F_S).$
Let $A_0=\{x\in A: g_0(x)\neq 0\}$ and $\tilde{A}_1=A.$ By Lemma \ref{lemma6.3},
for $\epsilon=\frac{\epsilon'}{2^l}$, $B_0=A_0$ and $B=\tilde{A}_l,$ there exists
$w_l\in \mathbb{C}^n$ such that
\begin{equation}\label{exp30}
m(\tilde{A}_l)<m(\tilde{A}_l\cup w_lA_0)<m(\tilde{A}_l)+\frac{\epsilon'}{2^l}.
\end{equation}
Put $\tilde{A}_{l+1}=\tilde{A}_l\cup w_lA_0$ and $\tilde A=\bigcup\limits_{l=1}^\infty\tilde{A}_l.$
Then $\tilde A_l$ is a non-decreasing sequence, and hence from (\ref{exp30}) it follows that
$m(\tilde{A}\smallsetminus A)<\epsilon'.$
For $l\in \mathbb{N},$ consider $g_l(z)=e^{\frac{i}{2}Im(z\cdot \bar{w_l})} g_0(z-w_l)$.
We show that $g_l\in \mathcal{R}( E_{\tilde A}\cap F_S)$ for each $l\in \mathbb{N},$ and
they are linearly independent. Let $\mathcal{B}_S$ be an orthonormal basis of $S.$ Then
we can extend  $\mathcal{B}_S$ to an orthonormal basis  $\mathcal{B}$ of $L^2(\mathbb{R}^n).$
For $\varphi \in L^2(\mathbb{R}^n)$ and $\psi\in \mathcal{B}\smallsetminus\mathcal{B}_S,$
we have
\begin{align*}
\langle W(g_l)\varphi,\psi\rangle &=\int_{\mathbb{C}^n}
g_l(z) \langle \pi(z)\varphi,\psi\rangle dz\\
&=\int_{\mathbb{C}^n} e^{\frac{i}{2}Im(z\cdot \bar{w_l})} g_0(z-w_l)
\langle \pi(z)\varphi,\psi\rangle dz \\
&=\int_{\mathbb{C}^n} e^{\frac{i}{2}Im(z\cdot \bar{w_l})} g_0(z)
\langle \pi(z+w_l)\varphi,\psi\rangle dz.
\end{align*}
Since $\pi(z)\pi(w)=e^{\frac{i}{2}Im(z\cdot\bar{w})}\pi(z+w),$ we get
\begin{align*}
\langle W(g_l)\varphi,\psi\rangle
&=\int_{\mathbb{C}^n} g_0(z)\langle \pi(z)\pi(w_l)\varphi,\psi\rangle dz\\
&=\int_{\mathbb{C}^n} g_0(z) \langle \pi(z)\tilde{\varphi},\psi\rangle dz\\
&=\langle W(g_0)\tilde{\varphi},\psi\rangle=0.
\end{align*}
Hence $\mathcal{R}(W(g_l))\subseteq\mathcal{B}_S.$ Let $A_l=A_{l-1}\cup w_lA_0.$
Then $\tilde{A}_{l+1}=\tilde{A}_l\cup A_l.$ Thus,
$m(A_l\smallsetminus A_{l-1})\geq m(\tilde{A}_{l+1}\smallsetminus \tilde{A}_l)>0.$
Let $s\in \mathbb{N}.$ Since, $A_s=A_0\cup w_1A_0\cup\cdots\cup w_s A_0$ and $g_l(z)=0$ on $(w_lA_0)^c$,
we have $E_{A_s}g_l=g_l$ for $l=0,1,\ldots,s$. Furthermore, $E_{A_s \smallsetminus A_{s-1}}g_l=0$ for
$l=0,\ldots,s-1$ and $E_{A_s \smallsetminus A_{s-1}}g_s\neq 0$. Therefore, it shows that $g_s$ is not a
linear combination of $g_0,\ldots,g_{s-1}.$ Since $s$ is arbitrary, $\{g_l:l\in \mathbb{N}\}$
is a linearly independent set in $\mathcal{R}( E_{\tilde A}\cap F_S).$
This proves the required proposition.
\end{proof}

\begin{proposition}\label{prop2}
Let $A$ be a measurable subset of $\mathbb{C}^n$ having finite Lebesgue measure.
Then the projection $E_A\cap F_N=0$.
\end{proposition}

\begin{proof}
In view of (\ref{exp24}) and Lemma \ref{lemma6}, we obtain
\[\dim\mathcal{R}( E_{\tilde{A}}\cap F_N) \leq (2\pi)^{-n}m(\tilde{A})N <\infty. \]
Therefore, as a corollary of Proposition \ref{prop4}, we get $E_A\cap F_N=0.$
\end{proof}

\begin{remark}
If $0<m(A)<\infty,$ then $\dim\mathcal{R}(E_A)=\infty.$ Now, in view of
Proposition \ref{prop2} and the fact that $E_A=(E_A\cap F_N) +(E_A\cap F_{N}^{\perp})=(E_A\cap F_{N}^{\perp}),$
it follows that $\dim\mathcal{R}(E_A\cap F_{N}^{\perp})=\infty.$ Since $m(A^c)=\infty,$
there exists a measurable set $B\subseteq A^c$ satisfying $0<m(B)<\infty.$ Hence
$\mathcal{R}(E_{A^c}\cap F_{N}^{\perp})\supseteq \mathcal{R}(E_B\cap F_{N}^{\perp})$.
This implies $\dim\mathcal{R}(E_{A^c}\cap F_{N}^{\perp})=\infty$.
Similarly, $\dim\mathcal{R}(E_{A^c}\cap F_N)=\infty$.
\end{remark}

The following theorem  is our main result of this section, which is analogous to
Benedicks-Amrein-Berthier theorem for the Heisenberg group.

\begin{theorem}\label{th7}
Let $A\subset\mathbb C^n$ be a set of finite Lebesgue measure. Suppose $f\in L^1(\mathbb{H}^n)$
and $\{(z,t)\in \mathbb{H}^n:f(z,t)\neq 0\} \subseteq A\times \mathbb{R}.$ If $\hat{f}(\lambda)$
is a finite rank operator for each $\lambda\in \mathbb{R}^\ast$, then $f=0$.
\end{theorem}

In order to prove Theorem \ref{th7}, it is sufficient to prove the following result
for the Weyl transform $W_\lambda.$

\begin{proposition}\label{prop1}
Let $g\in L^1(\mathbb{C}^n)$ and $\{z\in \mathbb{C}^n:g(z)\neq 0\} \subseteq A,$
where $m(A)$ is finite. Let $\lambda\in \mathbb{R}^*$ and $W_\lambda(g)$ has
finite rank. Then $g=0$.
\end{proposition}
Since $W_\lambda(g)$ is a finite rank operator, by the Plancheral theorem for the
Weyl transform, $g\in L^2(\mathbb{C}^n)$. Thus, it is enough to prove Proposition \ref{prop1}
for $g\in L^2(\mathbb{C}^n)$ and $\lambda=1.$ Note that Proposition \ref{prop1} follows from
Proposition \ref{prop2}.

\begin{remark}
As similar to the Euclidean set up, for strong annihilating pair (SAP)
\cite{HJ}, Proposition \ref{prop1} leads to the following concept of
strong annihilating pair for the Weyl transform, which appears in terms
of support of the function and rank of its Weyl transform.

\smallskip

{\em Let $A$ be a measurable subset of $\mathbb{C}^n$ and $P_S$ be the orthogonal
projection of $L^2(\mathbb{R}^n)$ onto a closed subspace $S$ of $L^2(\mathbb{R}^n).$
We call $(A,S)$ a SAP for the Weyl transform $W$ if there exists a positive
number $C=C(A,S)$ such that
\[\|g\|_2^2\leq C \left(\|g\|_{L^2(A^c)} +\|P_S^{\perp} W(g) \|_{HS}^2 \right)\]
for every $g \in L^2(\mathbb{C}^n).$}
\smallskip

In a similar way as to the Euclidean case, it can be shown that
if $A$ has finite measure and the dimension of $S$ is finite, then $(A,S)$
is a SAP for $W.$
\end{remark}

\begin{remark}\label{rk3} $(a)$
Let $G$ be a step two nilpotent Lie group (without MW condition). Then $G,$ as a set,
can be realized by $\mathbb{R}^{2n}\times\mathbb{R}^r\times\mathbb{R}^k.$ Let $\Lambda$
be the set of all irreducible unitary representation of $G$ those participate in the
Plancherel formula. Then $\Lambda$ can be identified with $\mathbb{R}^k\times\mathbb{R}^r.$
For unexplained details about step two nilpotent Lie group, we refer to \cite{CG,PS_00,PT,R}.

A similar proof as to the Heisenberg group leads to the following analogue of
Benedicks-Amrein-Berthier theorem on $G.$

\smallskip

\noindent{\em Statement: Let $A\subset\mathbb R^{2n}$ be a set of finite Lebesgue measure.
Suppose $f\in L^1(G)$ and $\{(x,y,t)\in G: f(x,y,t)\neq0\}\subseteq A\times\mathbb R^r\times\mathbb R^k.$ If $\hat{f}(\lambda,\mu)$
is a finite rank operator for each $(\lambda,\mu)\in\Lambda,$ then $f=0.$}
\end{remark}

$(b)$ Let $\mathcal E$ be a set of finite measure in $G.$ If $\{(x,y,t)\in G:f(x,y,t)\neq0\}\subseteq\mathcal E$
for some $f\in L^1(G),$ then it is natural to ask, whether there exists a nonzero function
$f\in L^1(G)$ such that $\hat{f}(\lambda,\mu)$ has finite rank for every choice of $\lambda,\mu\in\Lambda.$
Since $G\cong\mathbb R^{2n+r+k},$ if the projection of $\mathcal E$ into $\mathbb{R}^{2n}$ has
finite measure, then by Remark \ref{rk3}$(a)$ we get $f=0.$ However, the other case is still open.

\bigskip

\noindent{\bf Acknowledgements:}
The authors wish to thank E.K. Narayanan and Rama Rawat for some fruitful
suggestion and remarks. The first author gratefully acknowledges the support
provided by IIT Guwahati, Government of India.

\bigskip

%\small


\begin{thebibliography}{1000}

\bibitem{AVZ} M. L. Agranovsky, V. V. Volchkov and L. A. Zalcman, {\em Conical uniqueness sets for the
spherical Radon transform,} Bull. London Math. Soc. 31 (1999), no. 2, 231-236.

\bibitem{AB} W. O. Amrein and A. M. Berthier, {\em On support properties of {$L^{p}$}-functions and their
{F}ourier transforms,} J. Functional Analysis 24 (1977), no. 3, 258-267.

\bibitem{AAR} G. E. Andrews, R. Askey and R. Roy, {\em Special Functions,} Cambridge University Press, Cambridge, 1999.

\bibitem{A} D. H. Armitage, {\em Cones on which entire harmonic functions can vanish,}
Proc. Roy. Irish Acad. Sect. A 92 (1992), no. 1, 107-110.

\bibitem{Ba} D. B. Babot, {\em Heisenberg uniqueness pairs in the plane. Three parallel lines,}
Proc. Amer. Math. Soc. 141 (2013), no. 11, 3899-3904.

\bibitem{B} M. Benedicks, {\em On Fourier transforms of functions supported on sets of finite Lebesgue measure,}
J. Math. Anal. Appl. 106 (1985), no. 1, 180-183.

\bibitem{BD} A. Bonami and B. Demange, {\em A survey on uncertainty principles related to quadratic forms,}
Collect. Math. 2006, Vol. Extra, 1-36.

\bibitem{CHM} F. Canto-Mart\'in, H. Hedenmalm and A. Montes-Rodr\'iguez, {\em Perron-Frobenius operators and the
Klein-Gordon equation,} J. Eur. Math. Soc. (JEMS) 16 (2014), no. 1, 31-66.

\bibitem{CGGS}  A. Chattopadhyay, S. Ghosh, D. K. Giri and R. K. Srivastava, {\em Heisenberg uniqueness pairs on
the Euclidean spaces and the motion group,} C. R. Math. Acad. Sci. Paris 358 (2020), no. 3, 365-377.

\bibitem{CGS} A. Chattopadhyay, D. K. Giri and R. K. Srivastava, {\em Uniqueness of the Fourier transform on certain
Lie groups,} \href{https://arxiv.org/abs/1607.03832}{arXiv:1607.03832}.

\bibitem{CG} L. J. Corwin and F.P. Greenleaf, {\em Representations of nilpotent Lie groups and their applications,}
Cambridge Studies in Advanced Mathematics, 18, Cambridge University Press, Cambridge, 1990.

\bibitem{D} C. F. Dunkl, {\em Boundary value problems for harmonic functions on the Heisenberg group,}
Canad. J. Math. 38 (1986), no. 2, 478-512.

\bibitem{FGJ} A. Fern\'andez-Bertolin, K. Gr\"{o}chenig and  P. Jaming, {\em From Heisenberg uniqueness pairs
to properties of the Helmholtz and Laplace equations,} J. Math. Anal. Appl. 469 (2019), no. 1, 202-219.

\bibitem{FL} M. Filaseta and T.-Y. Lam, {\em On the irreducibility of the generalized Laguerre polynomials,}
Acta Arit. 105 (2002), 177-182.

\bibitem{FS} G. B. Folland and A. Sitaram, {\em The uncertainty principle: a mathematical survey,}
J. Fourier Anal. Appl. 3 (1997), 207-238.

\bibitem{GRR} D. K. Giri and R. Rawat, {\em Heisenberg uniqueness pairs for the hyperbola,} Bull. Lond. Math. Soc.,
doi: 10.1112/blms.12391 (to appear).

\bibitem{GR} D. K. Giri and R. K. Srivastava, {\em Heisenberg uniqueness pairs for some algebraic curves in the plane,}
Adv. Math. 310 (2017), 993-1016.

\bibitem{Gr1} P. C. Greiner, {\em Spherical harmonics on the Heisenberg group,} Canad. Math. Bull. 23 (1980), no. 4, 383-396.

\bibitem{GJ} K. Gr\"{o}chenig and  P. Jaming, {\em The Cram\'{e}r-Wold theorem on quadratic surfaces
and Heisenberg uniqueness pairs,} J. Inst. Math. Jussieu 19 (2020), no. 1, 117-135.

\bibitem{HJ} V. Havin and B. J\"{o}ricke, {\em The Uncertainty Principle in Harmonic Analysis,}
Ergebnisse der Mathematik und ihrer Grenzgebiete (3) [Results in Mathematics and Related Areas (3)],
28, Springer-Verlag, Berlin, 1994.

\bibitem{HR} H. Hedenmalm and A. Montes-Rodr\'{i}guez, {\em Heisenberg uniqueness pairs and the Klein-Gordon equation,}
Ann. of Math. (2) 173 (2011), no. 3, 1507-1527.

\bibitem{HR2} H. Hedenmalm and A. Montes-Rodr\'iguez, {\em The Klein-Gordon equation, the Hilbert transform, and dynamics
of Gauss-type maps,} J. Eur. Math. Soc. (JEMS) 22 (2020), no. 6, 1703-1757.

\bibitem{HR3} H. Hedenmalm and A. Montes-Rodr\'iguez, {\em The Klein-Gordon equation, the Hilbert transform, and Gauss-type
maps: $H^\infty$ approximation,} J. Anal. Math. (2020) (to appear).

\bibitem{H} A. J. Hogan, {\em A qualitative uncertainty principle for unimodular groups of type {${\rm I}$},}
Trans. Amer. Math. Soc. 340 (1993), no. 2, 587-594.

\bibitem{JK} P. Jaming and K. Kellay, {\em A dynamical system approach to Heisenberg uniqueness pairs,}
J. Anal. Math. 134 (2018), no. 1, 273-301.

\bibitem{L} N. Lev, {\em Uniqueness theorem for Fourier transform,}
Bull. Sci. Math. 135 (2011), no. 2, 134-140.

\bibitem{NR} E. K. Narayanan and P. K. Ratnakumar, {\em Benedicks' theorem for the Heisenberg group,}
Proc. Amer. Math. Soc. 138 (2010), no. 6, 2135-2140.

\bibitem{PS_00} S. Parui and R. P. Sarkar, {\em Beurling's theorem and $L^p-L^q$ Morgan's theorem for step two nilpotent
Lie groups,} Publ. Res. Inst. Math. Sci. 44 (2008), no. 4, 1027-1056.

\bibitem{PT} S. Parui and S. Thangavelu, {\em On theorems of Beurling and Hardy for certain step two nilpotent groups,}
Integral Transforms Spec. Funct. 20 (2009), no. 1-2, 127-145.

\bibitem{PS_0} V. Pati and A. Sitaram, {\em Some questions on integral geometry on Riemannian manifolds,}
Ergodic theory and harmonic analysis (Mumbai, 1999), Sankhya Ser. A 62 (2000), no. 3, 419-424.

\bibitem{PS} F. J. Price and A. Sitaram, {\em Functions and their Fourier transforms with supports
of finite measure for certain locally compact groups,} J. Funct. Anal. 79 (1988), no. 1, 166-182.

\bibitem{R} S. K. Ray, {\em Uncertainty principles on two step nilpotent Lie groups,}
Proc. Indian Acad. Sci. Math. Sci. 111 (2001), no. 3, 293-318.

\bibitem{Ru} W. Rudin, {\em Function theory in the unit ball of $\mathbb C^n$,} Springer-Verlag,
New York-Berlin, 1980.

\bibitem{SST} A. Sitaram, M. Sundari and S. Thangavelu, {\em Uncertainty principles on certain Lie groups,}
Proc. Indian Acad. Sci. Math. Sci. 105 (1995), no. 2, 135-151.

\bibitem{S1} P. Sj\"{o}lin, {\em Heisenberg uniqueness pairs and a theorem of Beurling and Malliavin,}
Bull. Sci. Math. 135 (2011), 125-133.

\bibitem{S2} P. Sj\"olin, {\em Heisenberg uniqueness pairs for the parabola,}
J. Fourier Anal. Appl. 19 (2013), no. 2, 410-416.

\bibitem{So} C. D. Sogge, {\em Oscillatory integrals and spherical harmonics,}
Duke Math. J. 53 (1986), no. 1, 43-65.

\bibitem{Sri3} R. K. Srivastava, {\em Real analytic expansion of spectral projections and extension of the
Hecke-Bochner identity,} Israel J. Math. 200 (2014), no. 1, 171-192.

\bibitem{Sri4} R. K. Srivastava, {\em  Non-harmonic cones are sets of injectivity for the twisted spherical means
on $\mathbb C^n$,} Trans. Amer. Math. Soc. 368 (2016), no. 3, 1941-1957.

\bibitem{Sri5} R. K. Srivastava, {\em  Non-harmonic cones are Heisenberg uniqueness pairs for the Fourier transform
on $\mathbb R^n$,} J. Fourier Anal. Appl. 24 (2018), no. 6, 1425-1437.

\bibitem{SW} E. M. Stein and G. Weiss, {\em Introduction to Fourier analysis on Euclidean spaces,}
Princeton Mathematical Series, No. 32. Princeton University Press (1971).

\bibitem{T1} S. Thangavelu, {\em  Harmonic analysis on the Heisenberg group,} Prog. Math., 159, Birkhäuser, Boston, 1998.

\bibitem{T2} S. Thangavelu, {\em An introduction to the uncertainty principle,} Prog. Math., 217, Birkhauser, Boston, 2004.

\bibitem{V} M. K. Vemuri, {\em Benedicks theorem for the Weyl transform,} J. Math. Anal. Appl. 452 (2017), no. 1, 209-217.

\bibitem{V1} F. J. G. Vieli, {\em A uniqueness result for the Fourier transform of measures on the sphere,}
Bull. Aust. Math. Soc. 86 (2012), no. 1, 78-82.

\bibitem{V2} F. J. G. Vieli, {\em A uniqueness result for the Fourier transform of measures on the paraboloid,}
Matematicki Vesnik 67 (2015), no. 1, 52-55.

\end{thebibliography}
\end{document}